\documentclass[11pt]{article}
\usepackage{setspace}
\usepackage{parskip} 

\usepackage{graphicx}%
\usepackage{multirow}%
\usepackage{amsmath,amssymb,amsfonts}%
\usepackage{amsthm}%
\usepackage{mathrsfs}%
\usepackage[title]{appendix}%
\usepackage{xcolor}%
\usepackage{textcomp}%
\usepackage{manyfoot}%
\usepackage{booktabs}%
\usepackage{algorithm}%
\usepackage{algorithmicx}%
\usepackage{algpseudocode}%
\usepackage{listings}%
\usepackage{mdframed}
\usepackage{tikz}
\usepackage{pgfplots}
\usetikzlibrary{arrows.meta}

\pgfmathdeclarefunction{f}{1}{\pgfmathparse{#1*#1*exp(-#1)}}

\usepackage{latexsym, amssymb, enumerate, amsmath,amsthm, amsfonts, amscd}
	
	\newtheorem{theorem}{Theorem}[section]
	
	\newtheorem{proposition}{Proposition}[section]

	\theoremstyle{definition}
	\newtheorem{remark}{{Remark}}[section]
	
	\newtheorem{example}{Example}[section]
	\newtheorem{definition}{Definition}[section]
	\def\R{\Bbb R}

	\def\N{\Bbb N}

	\def\dom\,{\mbox{\rm dom}}

	\def\diam{\mbox{\rm diam}}

	\def\cone{\mbox{\rm cone}}

	\newcommand{\uS}{\mathbb{S}}
	\newcommand{\mL}{\mathcal{L}}
	\newcommand{\mG}{\mathcal{G}}

	\usepackage[a4paper,left=2.2cm,right=2.2cm,top=2.3cm,bottom=2.3cm]{geometry}
	
	\usepackage{mdframed}
		\newsavebox{\myboxsb}
		\newenvironment{mybox}
		{\begin{lrbox}{\myboxsb}\begin{minipage}{\linewidth}}
				{\end{minipage}\end{lrbox}\fbox{\usebox{\myboxsb}}}
\begin{document} 

\begin{center}
{\sc\bf A New Notion of Tykhonov Well-Posedness for Optimization Problems}\\[1ex]
{\sc J. S. Chen\footnote{Department of Mathematics, National Taiwan Normal University, Taipei 116059, Taiwan (email: jschen@math.ntnu.edu.tw).}},
{\sc V. K. Hoang\footnote{Faculty of Mathematics and Computer
		Science, University of Science,
		Ho Chi Minh City, Vietnam.}$^,$\footnote{Vietnam National University, Ho Chi Minh City, Vietnam, (email: vokehoang@gmail.com).}} and
{\sc V. S. T. Long\footnotemark[2]$^,$\footnotemark[3]}
\end{center}
\vskip 0.4cm
\small{\bf Abstract.} 
Building upon the minimal time function, we propose and study a novel notion of {\em Tykhonov well-posedness with respect to a set of directions} for optimization problems. This concept generalizes the classical Tykhonov well-posedness by focusing on existence, stability and convergence along specific directions, rather than over the entire space. We first establish several characterizations of Tykhonov well-posedness with respect to a set of directions, formulated in terms of the diameter of level sets and admissible functions. We then investigate relationships between these level sets and admissible functions. 
To highlight the advantages of the proposed framework, we present several illustrative examples. In particular, we show that by selecting a suitable set of directions, 
optimization problems that are not well-posed in the classical sense may still be Tykhonov well-posed with respect to those directions. This viewpoint not only broadens the theoretical landscape of well-posedness but also has practical implications, as it allows numerical methods to be effectively adapted so that the generated sequences converge reliably to minimizers. 
\\[1ex]
		
\normalsize
\section{Introduction and Preliminaries}\label{intro}\vspace*{-0.1in}

The concepts of various types of well-posedness are crucial in nonlinear analysis, playing a central role in both theoretical frameworks and numerical methods. The fundamental results in this topic were established by
Hadamard \cite{hadamard1902problemes}.
In the early 1960s, Tykhonov \cite{tikhonov1966stability} studied well-posedness in the context of unconstrained minimization problems, defining it in terms of the existence and uniqueness of solutions, as well as the convergence of minimizing sequences to these solutions. This concept was later extended to constrained optimization problems by Levitin and Polyak \cite{levitin1966convergence}, who allowed minimizing sequences $x_k$ to lie outside the feasible set $A$, provided that the distance 
$d(x_k,A)$ (the distance from $x_k$
to $A$) approaches zero.
Since these fundamental works, numerous definitions of well-posedness have been introduced and extensively studied (see, e.g., \cite{anh2021tikhonov,bianchi2010well,canovas1999stability,crespi2007well,dontchev2006well,furi1970well,furi1970characterization,gutierrez2012pointwise,huang2006generalized,ioffe2005typical,lucchetti2020well,morgan2006discontinuous,sofonea2019well,zolezzi1981characterization} and the references therein). The main directions for developing Hadamard's, Tykhonov's and Levitin-Polyak's famous results so far are introducing generalized well-posesnesses and extending them to optimization-related problems in infinite dimensional spaces. It should be emphasized that in the aforementioned publications, most types of well-posedness have been
formulated from usual distance functions.

Among the most significant types of generalized distance functions, we mention the so-called minimal time function. Due to its significant applications in areas such as convex analysis, variational analysis, and optimization theory, this class of functions has been the subject of extensive and systematic research by many researchers. Relevant studies can be found in works like \cite{bardi1990approximation, colombo2010well, colombo2004subgradient, he2006subdifferentials, jiang2012subdifferential, mordukhovich2010limiting, mordukhovich2011subgradients, wolenski1998proximal}, which include numerous references and discussions on the topic.
One of the key advantages of the minimal time function is its ability to serve as a very general form of directional distance function.
Nam and Z\u{a}linescu \cite{nam2013variational} were 
crucial players in this process from the beginning. They studied several generalized differentiation properties of this class of functions and applied them to investigate location problems. 
The work was later significantly extended by Durea et al.~\cite{durea2016minimal, durea2017new}, who provided numerous applications in the theory of directional metric regularity. These developments paved the way for further research. For instance, Cibulka et al.~\cite{cibulka2020stability} investigated the stability of directional regularity, while Long et al.~\cite{long2021new, long2022invariant, long2023directional,KBL} introduced and analyzed new notions of directional error bounds and directional Levitin-polyak well-posedness for optimization problems, in which conventional distance functions were replaced by minimal time functions.  
In the same line, Chelmu{\c{s}} et al.~\cite{chelmucs2019directional} proposed the concept of directional Pareto solutions for set-valued constrained optimization problems, thereby generalizing the classical notion of Pareto efficiency. Further contributions to optimality conditions were provided by Ait Mansour et al.~\cite{ait2022strict} and Chelmu{\c{s}}~\cite{chelmus2022exact}. More recently, Chelmu{\c{s}} and Durea~\cite{chelmucs2021stability} obtained significant results on the stability of these directional Pareto solutions.  
It is important to emphasize that stability theory is intrinsically connected to well-posedness. Moreover, as shown in \cite{KBL}, the notion of directional Pareto solutions coincides with that of directional minimal solutions in scalar optimization problems.

Building on the research direction initiated by Durea et al.~\cite{durea2017new} and inspired by subsequent works such as \cite{burlicua2023directional, chelmucs2019directional, cibulka2020stability, long2022invariant, long2023directional, lucchetti1981characterization, nam2013variational}, 
it is natural to employ the minimal time function to introduce a new type of Tykhonov well-posedness for optimization problems. 
This framework allows for the unified treatment of both convex and non-convex cases. 
In particular, many practical optimization problems either lack global solutions or exhibit instability across all directions. 
However, the concept of Tykhonov well-posedness with respect to a set of directions makes it possible to identify stable solutions along specific directions. 
Such directionally stable solutions can yield valuable insights into the problem even in situations where classical Tykhonov well-posedness does not hold.

Our paper is structured as follows. Section \ref{sec:tykhonov} introduces a class of minimal time functions and proposes a new notion of Tykhonov well-posedness for optimization problems. Section \ref{sectionbasis} is devoted to establishing fundamental properties of minimal time functions, which will serve as technical tools for proving the main results. Section \ref{seccharecterization} presents characterizations of Tykhonov well-posedness with respect to a set of directions. We first provide results that do not require convexity of the feasible sets or the objective functions. Then, we formulate a directional version of Ekeland's variational principle and use it to obtain further characterizations under convexity assumptions. Finally, Section \ref{secrelasionship} investigates the properties of admissible functions and explores their relationships with level sets.


Throughout this paper (unless otherwise specified), the space $X$ under consideration is an arbitrary Banach space endowed with the norm $\|\cdot\|$, and $\langle \cdot,\cdot \rangle$ denotes the canonical pairing between $X$ and its topological dual $X^*$. The open unit ball and the unit sphere in $X$ are denoted respectively by $\mathbb B$ and $\mathbb{S}$. Given a set $C \subset X$, the diameter of $C$ is defined by
$$\diam(C) := \sup\{\|x-y\| \mid x,y \in C\}.$$
As usual, $\cone(C)$ denotes the cone generated by $C$, that is,
$$\cone(C):=\{\lambda x\mid \lambda \geq 0,\;x\in C\}.$$
The distance function from a point $x \in X$ to $C$ is given by
$$		d(x,C) := \inf\{\|x-y\|\mid y \in C\},$$
with the convention $d(x, \emptyset) := \infty$, i.e., $\inf \emptyset = \infty$. We also use $\mathbb{N}$ and $\mathbb{R}$ to denote the sets of natural numbers and real numbers, respectively.

\section{A new notion of Tykhonov well-posedness} 
\label{sec:tykhonov}

In this section, we introduce a new notion of Tykhonov well-posedness for optimization problems. We also present an example to highlight the significance of our study.

We begin by recalling the following class of functions, which plays a crucial role in our study.

\begin{definition}\label{def1}
	Let $M$ be a nonempt, closed and bounded subset of $X\setminus \mathbb B$ such that $\cone (M)$ is convex. Let $\Omega : X\to 2^X$ be a set-valued mapping. Then the function
	\begin{equation}\label{minimaltime}
		T_{M}(y,\Omega(x)):=\inf\{t\geq 0\mid (y+tM)\cap\Omega(x)\neq \emptyset\}\textrm{ for } (y,x)\in X\times X
	\end{equation}
	is called \textit{the minimal time function with respect to $M$.} By convention, we set $T_M (y, \emptyset):= \infty$  for every $y\in X$. 
\end{definition}

We do not consider general functions 
$T_{M}$ in this paper. Instead, we focus on the two cases below, which are sufficient for our purposes. For $\emptyset \neq O \subset X$, we define:

1) $T_M(y, O) := T_{M}(y,\Omega(x))$ when $\Omega(x)\equiv  O$ for all $x\in X$;

2) $T_M(y, x) := T_{M}(y, \Omega(x))$ when $\Omega$ is the identity map.

\begin{remark}
	
	It should be emphasized that the approach via the minimal time function, under the assumption that $M$ is a subset of the unit sphere $\mathbb{S}$ and $\Omega(x) \equiv O$ for all $x \in X$, has proven to be highly effective in studying several directional aspects of optimization (see, e.g., \cite{burlicua2023directional, chelmucs2019directional, cibulka2020stability, durea2016minimal, durea2017new, KBL, long2022invariant, long2023directional, lucchetti1981characterization, nam2013variational}). Furthermore, by considering other choices of $M$ and $\Omega$, the minimal time function has been investigated by many researchers for various purposes (see, e.g., \cite{colombo2010well, colombo2004subgradient, he2006subdifferentials, jiang2012subdifferential, mordukhovich2011applications, mordukhovich2011subgradients} and the references therein).
\end{remark}		
For any $x,y\in X$, the \textit{domains} of $T_M(\cdot,x)$ and $T_M(y,\cdot)$ are denoted and defined respectively by
$$\dom\,T_M(\cdot, x):=\{z\in X\mid T_M(z,x)<\infty\}$$
and
$$\dom\,T_M(y,\cdot):=\{w\in X\mid T_M(y,w)<\infty\}.$$

Let $f\colon X\to \mathbb{R}$ be a function. Consider the following optimization problem:
\[
\textrm{(MP)}\quad\text{minimize } f(x) \text{ subject to } x\in X.
\]		

Using the minimal time function \eqref{minimaltime}, we now introduce the notion of a global minimum point of $f$ with respect to $M$ as follows:

\begin{definition}\label{mind}  Consider the problem (MP) and the minimal time function given in Definition \ref{def1}. We say that $\bar x\in X$ is \textit{a global 
		minimum point of $f$ with respect to (the set of directions) $M$} on $D_{\bar x}$ if and only if
	\[
	f(\bar x)\leq f(x)\text{ for all } x\in \dom\, T_M(\cdot,\bar x).
	\]
\end{definition}

For the sake of simplicity, throughout the paper we set

$$D_x := \dom\, T_M(\cdot, x) \quad \text{for all } x \in X.$$

\begin{remark}
	Clearly, if $M = \mathbb{S}$, then global minimum points of $f$ with respect to $M$ coincide with the usual global minimum points. Moreover, if $M$ is a proper subset of the unit sphere $\mathbb{S}$, then global minimum points of $f$ with respect to $M$ reduce to global directional minimum points of $f$ as defined in \cite[Definition~2]{KBL}.  
	For a more general concept of directional Pareto minimality for set-valued mappings with respect to a set of directions $M$, we refer the interested reader to \cite{ait2022strict,chelmus2022exact,chelmucs2021stability,chelmucs2019directional}.
\end{remark}

As a straightforward extension of the usual notion of Tykhonov well-posedness, we may naturally propose the following.

\begin{definition}\label{wellposedd} 	The problem (MP) is said to be \textit{Tykhonov well-posed with respect to $M$} if and only if 
	\begin{enumerate}
		\item [(i)] there exists $\bar x\in X$ such that it is the unique global minimum point of $f$ with respect to $M$ on $D_{\bar x}$, and 
		\item [(ii)] for any sequence $(x_k)\subset D_{\bar x}$ satisfying $f(x_k)\to f(\bar x)$, one has~$T_M(x_k,\bar x)\to 0$ as $k\to\infty$. 
	\end{enumerate}
\end{definition}

\begin{remark}
	\begin{enumerate}
		\item [{\rm (i)}] To the best of our knowledge, the concept of Tykhonov well-posedness with respect to $M$ has not yet been addressed in the literature. When $M=\uS$ (the unit sphere), this notion reduces to the classical Tykhonov well-posedness; see \cite[pp. 1]{tikhonov1966stability}.
		
		\item [{\rm (ii)}]	The notion of Tykhonov well-posedness with respect to $M$ is useful for problems where global solutions may lack stability or may not be well-defined, while solutions along specific directions demonstrate well-posedness properties. These aspects will be illustrated in the following examples.
	\end{enumerate}
	
\end{remark}


\begin{example}
	Let $X=\ell^2$, the space of square summable sequences of real numbers, and let $f: X\to \mathbb{R}$ be given by
	$$
	f(x):=\left\{\begin{array}{ll}
		\sum_{i=1}^{\infty} x^{(i)} & \text { if } x^{(i)}\neq 0 \text{ for at most finitely many }i, \\
		1 & \text { otherwise, }
	\end{array} \right.
	$$
	where $x^{(i)}$ denotes the $i$-th coordinate of $x\in \ell^2$. Since $f$ is not bounded below on X, problem (MP) has no usual solution.
	
	Now let $\bar x=(0,0,...)$ and $M=\{-e_1=(-1,0,...)\}$. Then we obtain from Proposition~2.2(i) in \cite{durea2016minimal} that $$D_{\bar x}=\bar x-\cone(M)=\{(\alpha,0,...)\mid \alpha \geq 0\},$$
	and hence
	$$f(\bar x)=0<\alpha_x =\sum_{n=1}^{\infty}x^{(i)}=f(x)\text{  for all }x=(\alpha_x,0,...)\in D_{\bar x}\setminus\{\bar x\},$$
	which means $\bar x$ is the unique global minimum of $f$ with respect to~$M$ on $D_{\bar x}$. Next, take any sequence $(x_k)$ in $D_{\bar x}$. Then we have $x_k=(\alpha_k,0,..)$ for some $\alpha_k\geq 0$. Since $x_k\in \dom\, T_M(\cdot,\bar x)$, it follows from Proposition 2.3 (i) in \cite{durea2016minimal} that
	\begin{equation}\label{key22}
		T_M(x_k,\bar x)=\|\bar x-x_k\|=\alpha_k.
	\end{equation}
	Moreover, if $f(x_k)\to f(\bar x)$, then $\alpha_k\to 0$.
	This together with (\ref{key22}) implies that $T_M(x_k,\bar x)\to 0$. Therefore, problems (MP) is Tykhonov well-posed with respect to $M$. 
\end{example}

\section{Elementary Properties of minimal time functions}
\label{sectionbasis}

This section is devoted to establishing several fundamental properties of certain special cases of the minimal time function~\eqref{minimaltime}, which are well suited to our aims in the subsequent sections. 

Consider the set $M$ given in Definition \ref{def1} and put
$$\|M\|:=\sup\{\|u\|\mid u\in M\}.$$
Under the conditions on the set $M$ imposed in Definition \ref{def1}, we obtain  $$1\leq \|M\|<\infty.$$	

We begin with the following proposition, some steps in the proof of which are inspired by the arguments used in \cite{durea2016minimal,mordukhovich2010limiting,mordukhovich2011subgradients}.

\begin{proposition}\label{remark1}
	Let $x, y \in X$. Then the following statements hold:
	\begin{enumerate}
		\item[{\rm (i)}] The domains of $T_M(\cdot,x)$ and $T_M(y,\cdot)$ are given, respectively, by
		\[
		\dom\, T_M(\cdot, x) = x - \operatorname{cone}(M)
		\quad \text{and} \quad
		\dom\, T_M(y, \cdot) = y + \operatorname{cone}(M),
		\]
		and both sets are closed.
		\item[{\rm (ii)}] For every $z \in \dom\, T_M(\cdot, x)$, there exists an element $u_z\in M$ such that
		
		\begin{equation}\label{dangthuc1}
			z=x+u_zT_M(z,x).
		\end{equation}
		Similarly, for every $w \in \dom\, T_M(y,\cdot)$, there exists an element $u_w\in M$ such that
		\[
		w = y+ u_w\,T_M(y,w).
		\]
	\end{enumerate}
\end{proposition}

\begin{proof}
	We proceed with the proofs of the first assertions of (i) and (ii) while observing that the verification of the second assertions of (i) and (ii) are similar. 
	
	(i) Pick any $z \in \dom\,T_M (\cdot,x)$. Then, we have $T_M (z,x) <\infty$. By Definition \ref{def1}, there are a number $t \geq 0$ and $u \in M$ such that $z + tu=x$. Consequently, 
	$$z=x - tu \in  x- \cone (M ).$$
	The reverse inclusion is obvious. Moreover, since $0\notin M$ and $M$ is a closed and bounded set, it follows from \cite[Theorem 2.5(i)]{durea2013} that $\cone(M)$ is closed. Therefore, $\dom\,T_M(\cdot,x)$ is also closed. This completes the proof of the first assertion of (i). 
	
	(ii) Fix any $z\in \dom\,T_M(\cdot,x).$ Then there exist sequences $(t_k) \subset [0,\infty)$ and $(u_k )\subset M$ such that $$z + t_ku_k=x\quad\text{and}\quad t_k\to T_M(z,x)\text{ as } k\to \infty
	.$$ There are two possible cases: either $T_M(z,x)=0$ or $T_M(z,x)>0$. In the case where $T_M(z,x)=0$, we have
	$$\|z-x\|=t_k\|u_k\|\leq t_k\|M\|.$$
	Since $\|M\|<\infty$ and $t_k\to T_M(z,x)= 0$, if follows that $\|z-x\|= 0$, 
	and hence $z=x$. This implies that the equality \eqref{dangthuc1} holds. In the second case, we get for $k$ sufficiently large that
	$$\frac{x-z}{t_k}=u_k\in M\quad \text{and}\quad u_k\to u_z:=\frac{x-z}{T_M(z,x)}\text{ as } k\to\infty.$$ 
	Since $M$ is closed, we obtain $u_z\in M$, which again yields \eqref{dangthuc1}. The proof of the proposition is complete. 
\end{proof}

\begin{remark} \label{vdkhac}
	If $M$ is a subset of $\mathbb S$, it follows from Proposition \ref{remark1}(ii) that for every $z\in \dom\,T_M(\cdot,x)$, we have
	$$\|z-x\|=T_M(z,x).$$
	However, in general, $T_M(z,x)$ may differ from $T_M(x,z)$, as illustrated by the following simple example.  
\end{remark}

\begin{example} \label{vdratkhac}
	Let $X=\mathbb{R}$, $M=\{-1\}$ and $x=2$. By Proposition \ref{remark1}(i), we have
	$$\dom\,T_M(\cdot,x)=x-\cone(M)=[2,\infty).$$
	Take $z=3$. Then it follows from Definition \ref{def1} that
	$$T_M(3,2)=\inf\{t\geq 0\mid 3-t=2\}=1=\|3-2\|.$$
	On the other hand, 
	$$T_M(2,3)=\inf\{t\geq 0\mid 2-t=3\}=\infty.$$
	Thus, $T_M(3,2)\neq T_M(2,3)$.
\end{example}

The next results will serve as essential tools for proving the main theorems of the paper.		

\begin{proposition}\label{lemmato0}
	For every $\epsilon >0$ and $\emptyset\neq O \subset X$, we put
	$$O^{\epsilon} := \{x\in X\mid T_M(x,O)\leq \epsilon\}.$$
	If $\diam (O)\to 0$ as $\epsilon\to 0$, then we have $\diam (O^{\epsilon})\to 0$ as $\epsilon\to 0$. 
\end{proposition}
\begin{proof}
	For every $\epsilon > 0$, since $T_M(x_\epsilon,O)\leq \epsilon$ and $T_M(y_\epsilon,O)\leq \epsilon$ for all $x_\epsilon,y_\epsilon\in O^{\epsilon}$, there exist $v_\epsilon,w_\epsilon\in O$ such that 
	\begin{equation}\label{keyepsilon}
		T_M(x_\epsilon,v_\epsilon)<2\epsilon\quad\text{and}\quad T_M(y_\epsilon,w_\epsilon)<2\epsilon.	
	\end{equation}
	Then it follows from Proposition \ref{remark1}(ii) that there exist $u_x,u_y\in M$	such that		 
	$$\begin{array}{ll}
		\|x_\epsilon-y_\epsilon\|&\leq \|x_\epsilon-v_\epsilon\| + \|v_\epsilon - w_\epsilon\| + \|w_\epsilon - y_\epsilon\|\\[0.1in]
		&= \|u_x\|T_M(x_\epsilon,v_\epsilon)  + \|v_\epsilon-w_\epsilon\|+\|u_y\|T_M(y_\epsilon,w_\epsilon)\\[0.1in]
		&\leq 4\|M\|\epsilon + \diam (O)\quad\text{(by \eqref{keyepsilon})}
	\end{array}$$
	Since $\|M\|<\infty$ and $\diam (O)\to 0$ as $\epsilon\to 0$, it follows that $\diam (O^{\epsilon})\to 0$ as $\epsilon\to 0$.
\end{proof}

We now recall a classical property of the minimal time function, namely its triangle inequality.

\begin{proposition}\label{long2023directional}
	For any $x,y,z\in X$, we have
	\begin{equation}\label{triangle}
		T_M(z,x)\leq T_M(z,y)+T_M(y,x).
	\end{equation}
\end{proposition}
\begin{proof} The proof is similar to the arguments of results in \cite{durea2017new,long2022invariant}. However, we prefer to give the details, for the sake of readability
	
	If either $T_M(z,y) = \infty$ or $T_M(y,x) = \infty$, the inequality \eqref{triangle} is trivial. So  
	assume both are finite. By Proposition \ref{remark1}(ii), there exist $u_1, u_2 \in M$ such that
	\[
	y  =z+ T_M(z,y)\,u_1, \quad\text{and}\quad x  = y+T_M(y,x)\,u_2.
	\]
	Therefore,
	\[
	x = z+T_M(z,y)\,u_1 + T_M(y,x)\,u_2.
	\]
	Since $\operatorname{cone}(M)$ is convex, it follows that $$u := \frac{T_M(z,y)}{T_M(z,y)+T_M(y,x)} u_1 + \frac{T_M(y,x)}{T_M(z,y)+T_M(y,x)} u_2 \in \operatorname{cone}(M),$$ and hence $x$ can be expressed as:
	\[
	x = z+ [T_M(z,y)+T_M(y,x)] u.
	\]
	By the definition of the minimal time \eqref{minimaltime}, we obtain the desired inequality:
	\[
	T_M(z,x) \leq \ T_M(z,y) + T_M(y,x).\] 
\end{proof}	

A basic inclusion property of the domains of the minimal time function can be stated as follows:	

\begin{proposition}\label{lemma2.2}
	For any $x \in X$, if $y \in \dom\, T_M(\cdot,x)$, then  
	\[
	\dom\, T_M(\cdot,y) \subset \dom\, T_M(\cdot,x).
	\]
\end{proposition}

\begin{proof} Fix any $y\in\dom\,T_M(\cdot,x)$.
	Take an arbitrary $z \in \dom\, T_M(\cdot,y)$. Then, by Proposition~\ref{long2023directional}, we have
	\[ T_M(z,x) \leq \ T_M(z,y) + T_M(y,x) < \infty,\] 
	which yields $z \in \dom\,T_M(\cdot,x)$.  
	Since $z$ was chosen arbitrarily from $\dom\, T_M(\cdot,y)$, we conclude that  
	\[
	\dom\, T_M(\cdot,y) \subset \dom\, T_M(\cdot,x).
	\]
\end{proof}

To end this section, we provide a property concerning the special case of the minimal time function \eqref{minimaltime}. 

\begin{proposition}\label{pronew}
	Consider the minimal time function defined in Definition \ref{def1}. Suppose in addition that $M$ is a subset of the unit sphere $\mathbb{S}$.
	Then, for every $t\in[0,\infty)$, there exists $x\in D_{\bar x}$ such that $$T_M(x,\bar x)=t.$$
\end{proposition}

\begin{proof}
	If $t=0$, simply take $x=\bar x$. Then $T_M(\bar x, \bar x)=0.$ Suppose $t>0$. Choose any $ u^\ast\in M$ and set 
	\[
	x:=\bar x - t\,u^\ast.
	\]
	Then $\bar x = x + t\,u^\ast \in x + tM$, so by definition $T_M(x,\bar x)\le t$.
	
	Suppose, to the contrary, that $T_M(x,\bar x)=s<t$. Then it follows from Proposition \ref{remark1}(ii) that there exists $u\in M$ such that
	\[
	x = \bar x - s u,\] 
	which implies
	\[s\,u=t\,u^\ast .
	\]
	Taking norms gives
	\[
	\|u\|=\frac{t}{s}\,\|u^\ast\|=\frac{t}{s}>1,
	\]
	contradicting the assumption that $M\subset \mathbb{S}$. Therefore, $T_M(x,\bar x)=t$.
\end{proof}

\section{Characterizations of Tykhonov well-posedness with respect to~$M$} \label{seccharecterization}

In this section, we present characterizations of Tykhonov well-posedness with respect to $M$ for both convex and nonconvex optimization problems. Moreover, the examples below illustrate that Tykhonov well-posedness with respect to $M$ offers a valuable perspective on optimization and analysis.

\subsection{Well-posedness without convexity}
We begin with a characterization of Tykhonov well-posedness with respect to $M$ for problem (MP) in terms of level sets of $f$. It is often called the Furi--Vignoli criterion.

\begin{theorem}\label{theo1}
	Consider the problem {\rm (MP)} and the minimal time function \eqref{minimaltime}. Suppose that $f$ is lower semicontinuous.
	Then ${\rm (MP)}$ is Tykhonov well-posed with respect to $M$ if and only if there exists $\bar x\in X$ such that f is bounded from below on $D_{\bar x}= \dom\, T_M(\cdot,\bar x)$ and  
	\begin{equation}\label{eqthm1}
		\diam(\mathcal{L}(\bar x,\epsilon))\to 0\quad\text{as }\epsilon\to 0,
	\end{equation}
	where $\epsilon >0$ and 
	\begin{equation}\label{levelset}
		\mL(\bar x,\epsilon):=\left\{y\in D_{\bar x}\Big | f(y)\leq \inf_{x\in D_{\bar x}}f(x)+\epsilon\right\}.
	\end{equation}
\end{theorem}
\begin{proof}
	If ${\rm (MP)}$ is Tykhonov well-posed with respect to $M$, then there exists $\bar x\in X$ such that $$f(\bar x) < f(x)\quad\text{for all }x\in D_{\bar x}\setminus\{\bar x\},$$
	which implies that  $f$ is bounded from below on $D_{\bar x}.$  Next, we proceed by contradiction. Suppose that 
	there exists a sequence $(\epsilon_k)$ 
	such that $$\diam(\mL(\bar x,\epsilon_k))\not\to 0\quad\text{as }\epsilon_k \to 0.$$ 
	Then we can select a subsequence of $(\epsilon_k)$ (without relabeling) and a number $a>0$ such that $$\diam(\mL(\bar x,\epsilon_k))>2a.$$ 
	Consequently, for every $k\in \N$, there exists $x_k\in \mL(\bar x,\epsilon_k)$ such that $\|x_k-\bar x\|\geq a$ and 
	\[
	f(\bar x) < f(x_k)\leq \inf_{x\in D_{\bar x}} f(x) + \epsilon_k = f(\bar x) + \epsilon_k.
	\]
	This yields $f(x_k)\to f(\bar x)$ as $k\to\infty$. However, since $x_k\in \dom\, T(\cdot,\bar x)$ and $\|x_k-\bar x\|\geq a$, it follows from Proposition~\ref{remark1}(ii) that there exists $u_k\in M$ such that
	\[
	T_M(x_k,\bar x)= \frac{\|x_k-\bar x\|}{\|u_k\|}\geq \frac{a}{\|M\|}>0\quad\text{for all }k\in \mathbb N.
	\]
	This contradicts condition (ii) of Definition \ref{wellposedd}. Therefore, 
	$$\diam(\mL(\bar x,\epsilon_k))\to 0\quad\text{as }\epsilon_k \to 0.$$ 
	
	Conversely, suppose that there exists $\bar x\in X$ such that $f$ is bounded from below on $D_{\bar x}$ and $$\diam(\mL(\bar x,\epsilon))\to 0\quad\text{as } \epsilon \to 0.$$ 
	Then, $\inf\limits_{x\in D_{\bar x}}f(x)\in\R$, and hence for every $k\in\N$, there exists $x_k\in D_{\bar x}$ such that
	\[
	f(x_k)<\inf_{x\in D_{\bar x}}f(x)+\frac{1}{2^{k}}.
	\]
	Note that if $y\in \mL(\bar x,1/2^{k+1})$ then 
	\[
	f(y)\leq \inf_{x\in D_{\bar x}}f(x)+\frac{1}{2^{k+1}}<\inf_{x\in D_{\bar x}}f(x)+\frac{1}{2^{k}}.
	\]
	This implies that $\mL(\bar x,1/2^{k+1})\subset \mL(\bar x,1/2^{k})$ for all $k\in \N$.
	Moreover, since $f$ is lower semicontinuous, the set $\mL(\bar x,1/2^{k})$ is closed for all $k\in\N$. On the other hand, by the hypothesis, we have $\diam\left(\mL(\bar x,1/2^{k})\right)\to 0$ as $k\to \infty$. Thus, the assumptions of Cantor's intersection theorem are satisfied. Applying this theorem, we obtain an element $\bar y \in D_{\bar x}$ such that
	\begin{equation}\label{unique}
		\bigcap_{k\in \N}\mL(\bar x,1/2^{k})=\{\bar y\}.		
	\end{equation}
	By the definition of the level set $\mL(\bar x,1/2^{k})$, one has $$f(\bar y)\leq \inf_{x\in D_{\bar x}}f(x)+\frac{1}{2^k}\text{ for all }k\in \N,$$ and hence $f(\bar y)\leq \inf\limits_{x\in D_{\bar x}}f(x)$. Furthermore, since $\bar y \in D_{\bar x}$, it follows that 
	\begin{equation}\label{key1}
		f(\bar y) = \inf_{x\in D_{\bar x}}f(x).
	\end{equation}
	Now, we show that (MP) is Tykhonov well-posed with respect to~$M$. Indeed, since $\bar y\in D_{\bar x}$, it follows from Proposition \ref{lemma2.2} that
	\[
	D_{\bar y} = \dom\, T_M(\cdot,\bar y)\subset \dom\, T_M(\cdot,\bar x) = D_{\bar x}.
	\]
	Combining this with (\ref{key1}), we obtain  
	\[
	\inf_{y\in D_{\bar y}}f(y)\leq f(\bar y) = \inf_{x\in D_{\bar x}}f(x)\leq\inf_{y\in D_{\bar y}}f(y),
	\]
	which yields $$f(\bar y)=\inf\limits_{y\in D_{\bar y}}f(y).$$ To verify that $\bar y$ is the unique global minimum
	point of $f$ with respect to $M$ on $D_{\bar y}$, we argue by contradiction. Suppose that there exists $\bar z\in D_{\bar y}\setminus\{\bar y\}$ such that $f(\bar z)=f(\bar y)$. Then it follows from (\ref{key1}) and the definition of $\mL(\bar x,1/2^k)$ that $$\bar z\in \mL(\bar x,1/2^k)\text{ for all }k\in\mathbb{N}.$$ Combining this with (\ref{unique}) gives a contradiction. This yields that condition (i) of Definition \ref{wellposedd} holds.
	It remains to verify condition (ii) of Definition \ref{wellposedd}. Specifically, we need to show that for any sequence $(y_k)\subset D_{\bar y}\subset D_{\bar x}$, if $f(y_k)\to f(\bar y)$ then 
	\begin{equation}\label{keynew6}
		T_M(y_k,\bar y)\to 0\quad \text{as }k\to \infty. 
	\end{equation}
	Indeed, since $f(\bar y)=\inf\limits_{x\in D_{\bar x}}f(x)$, we obtain from \eqref{levelset} that $$\bar y, y_k\in \mL(\bar x, f(y_k)-f(\bar y))\text{ for all }k\in\N.$$ 
	Therefore, 
	\begin{equation}\label{key17}
		\|y_k-\bar y\| \leq \diam(\mL(\bar x, f(y_k)-f(\bar y))\text{ for all }k\in\N. 
	\end{equation}
	Since
	$f(y_k)-f(\bar y)\to 0$ as $k\to \infty$, it follows from \eqref{eqthm1} that 
	$$\diam\big(\mL(\bar x, f(y_k)-f(\bar y))\big)\to 0\quad\text{as }k\to~\infty.$$ 
	Combining this with (\ref{key17}) gives $$\|y_k-\bar y\|\to~0\quad\text{as }k\to \infty.$$ 
	By Proposition \ref{remark1}(ii), this further implies that there exists $u_k\in M$ such that 
	$$\|u_k\| T_M(y_k,\bar y)=\|y_k-\bar y\|\to 0\quad\text{as } k\to \infty.$$
	Under the assumptions on the set $M$ given in Definition \ref{def1}, we have  $$1\leq \|u_k\|<\infty.$$ Consequently, the relation \eqref{keynew6} holds. This justifies the reverse implication and completes the proof of the theorem.
\end{proof}

\begin{remark} Unlike \cite[Theorem~1]{KBL}, where only a sufficient condition was given, condition \eqref{eqthm1} yields a complete characterization of Tykhonov well-posedness with respect to $M$. Moreover, if $M=\mathbb{S}$, then Theorem \ref{theo1} becomes Theorem 2.2 in \cite{furi1970well} (for the case of Banach
	spaces). 
\end{remark}
The following classical example illustrates the advantages of Theorem \ref{theo1} by showing that while global solutions may lack stability in all directions, Tykhonov well-posedness with respect to $M$ enables stability to be attained in specific directions.

\begin{example}\label{exht}	Let $X=\R$ and let $f\colon X\to \R$ be defined by $f(x)=x^2 e^{-x}$. Then, (MP) is not Tykhonov well posed, since the sequence $(x_k)=(k)$ is minimizing but it does not converge to the unique minimum $0$.
	
	To try with Theorem \ref{theo1}, we take $M=\{3\}$ and $\bar x = \frac{1}{2}$. It is easy to see that $f$ is bounded on $D_{\bar x}=\bar x - \cone (M) = (-\infty,1/2]$. Next, we prove that for any $\epsilon >0$,  
	\begin{equation}\label{eq0}
		\mL\left(1/2,\epsilon\right)=\left\{x\in (-\infty,1/2]\mid x^2e^{-x}\leq \epsilon\right\}\subset (-\sqrt[3]{\epsilon},\sqrt[3]{\epsilon}).
	\end{equation}
	Indeed, for each $x\in \mL\left(1/2,\epsilon\right)$, a direct verification shows that $$-x^3<x^2e^{-x}\leq \epsilon,$$ and thus
	$	-\sqrt[3]{\epsilon}<x.$
	Moreover, if $x\in \mL(1/2,\epsilon)$ then we also have $$x^3<x^2e^{-x}\leq\epsilon,$$ which implies
	$		x<\sqrt[3]{\epsilon},$	and therefore (\ref{eq0}) holds.
	It follows from (\ref{eq0})  that $\diam(\mL(1/2,\epsilon))\to 0$ as $\epsilon\to 0$. According to Theorem \ref{theo1}, the problem (MP) is Tykhonov well-posed with respect to $M$. It is interesting to note that $\bar x=1/2$ is not a minimum point of $f$ on $(-\infty,1/2]$.
\end{example}

\begin{remark}
	One of the key advantages of well-posedness with respect to $M$ is that it facilitates suitable adaptations of numerical algorithms (when necessary) to ensure that the generated sequences converge reliably to the unique solution. To illustrate this advantage, we now apply Newton's method to Example~\ref{exht}.

	As seen in Example \ref{exht}, if a sequence generated by Newton's method is minimizing and contained in $(-\infty,1/2]$, then it converges to the unique minimum point $0$ with respect to $M$ (the right direction). However, sequences that lie outside $(-\infty,1/2]$, may not converge to $0$. Indeed, we can compute directly that
	$$f'(x)=e^{-x}(2x-x^2)\text{ and } f''(x)=e^{-x}(x^2-4x+2).$$
	Consider the iteration: 
	\begin{equation}\label{key777}
		x_{k+1}:=x_k+p_k=x_k+\frac{x^2_k-2x_k}{x^2_k-4x_k+2}=\frac{x^2_k(x_k-3)}{(x_k-2)^2-2},
	\end{equation}
	where $p_k:=-\frac{f'(x_k)}{f''(x_{k})}$ is called the Newton direction. 
	
	Starting with $x_1=0.6$, we observe that the sequence $(x_k)$, produced by Newton's Method, is contained in $ [1,\infty)$. Despite $f(x_k)$ decreasing rapidly starting from the second iteration $(f(x_2)\approx 1.941* 10^{-7}$), the sequence $(x_k)$ does not converge to the minimum point because $x_k<x_{k+1}$ for all $k$.

	Nevertheless, we can choose any arbitrary starting point $y_1\leq 1/2$, even if it is far from the minimum point, such as $y_1=-100$. Then the relation~\eqref{key777} guarantees the generation a sequence $y_k\in  (-\infty,0]$ for all $k\geq 2$. Moreover, since $f''(y_k)>0$ for all $y_k\leq 0$, the corresponding Newton direction $p_k$ is a descent direction (see, e.g., \cite{optimnumerical}). Therefore, the sequence $(y_k)$ converges to the minimum point with respect to $M$, as shown in Example \ref{exht}, see Figure \ref{figure1}. 		
\end{remark}

\begin{figure} 
	\begin{mybox}
		\centering
		\begin{tikzpicture}
			\begin{axis}[
				width=12cm, height=7cm,
				axis lines=middle,
				xlabel={$x$}, ylabel={$f(x)=x^2 e^{-x}$},
				grid=none,
				xmin=-8, xmax=10, ymin=-1, ymax=4,
				domain=-8:10, samples=350, smooth,
				restrict y to domain*=-1:5,
				unbounded coords=discard,
				xticklabels=\empty, yticklabels=\empty, tick style={draw=none},
				]
				\addplot[very thick] {x^2*exp(-x)};
				
				\addplot+[only marks, mark=*, mark size=2.5pt,
				mark options={fill=red, draw=red}] coordinates {(0.6,0)};
				\node[below] at (axis cs:0.6,0) {$x_1$};
				
				\addplot+[only marks, mark=*, mark size=2.5pt,
				mark options={fill=red, draw=red}] coordinates {(7,0)};
				\node[below] at (axis cs:7,0) {$x_k$};
				
				\addplot+[only marks, mark=*, mark size=2.5pt,
				mark options={fill=red, draw=red}] coordinates {(9,0)};
				\node[below] at (axis cs:9,0) {$x_{k+1}$};
				
				\addplot+[only marks, mark=*, mark size=2.5pt,
				mark options={fill=blue, draw=red}] coordinates {(-1.5,0)};
				\node[below] at (axis cs:-1.5,0) {$y_k$};
				
				\addplot+[only marks, mark=*, mark size=2.5pt,
				mark options={fill=blue, draw=red}] coordinates {(-0.5,0)};
				\node[below] at (axis cs:-0.52,0) {$y_{k+1}$};
			\end{axis}
		\end{tikzpicture}
	\end{mybox}
	\caption{Illustration of the sequences $\{x_k\}$ and $\{y_k\}$ generated by Newton's method.}
	\label{figure1}	
\end{figure}
\vskip 1cm

To proceed, we recall the concept of \textit{admissible} (also called \textit{forcing}) \textit{functions} (see, e.g., \cite[pp.~5--6]{dontchev2006well}).
\begin{definition}\label{admisd} Let $D\subset [0,\infty)$ be such that $0\in D$.
	A function $c\colon D\rightarrow [0,\infty)$ is called admissible if and only if
	\begin{enumerate}
		\item [{\rm(i)}] $c(0)=0$ and
		\item [{\rm(ii)}] for any sequence $(t_k)\subset D$ satisfying $c(t_k)\rightarrow 0$, one has $t_k\rightarrow 0$.
	\end{enumerate} 
\end{definition}

Without the lower semicontinuity of $f$, the Tykhonov well-posedness with respect to $M$ of problem (MP) can be characterized by the above admissible function as follows:

\begin{theorem}\label{theo2}
	Problem (MP) is Tyknonov well-posed with respect to $M$ if and only if there exists $\bar x\in X$ and an admissible function $c$ such that  
	\begin{equation}\label{key23}
		f(\bar x)+c(T_M(x,\bar x))\leq f(x)\quad\text{for all }x\in D_{\bar x}.
	\end{equation}
\end{theorem}
\begin{proof}
	If (MP) is Tykhonov well-posed with respect to $M$, then there exists $\bar x\in X$ such that 
	\begin{equation}\label{keynew8}
		f(\bar x) < f(x)\quad\text{for all }x\in D_{\bar x}\setminus \{\bar x\}.
	\end{equation} 
	Put $$D:=\{T_M(x,\bar x)\mid x\in D_{\bar x}\}.$$
	Clearly, $D\subset [0,\infty)$, and since $T_M(\bar x,\bar x)=0$, we have $0\in D$.
	For every $t\in  D$,  define the function $c$ by 
	\[
	c(t):=
	\inf\left\{f(y)-f(\bar x)\mid y\in D_{\bar x}\textrm{ and } T_M(y,\bar x)=t\right\}.
	\] 
	Then for every $x\in D_{\bar x}$, we have 
	$$\begin{array}{ll}
		c\left(T_M(x,\bar x)\right)&=\inf\left\{f(y)-f(\bar x)\mid y\in D_{\bar x}\textrm{ and } T_M(y,\bar x)=T_M(x,\bar x)\right\}\\
		&\leq f(x)-f(\bar x),
	\end{array}$$ 
	which implies that the inequality (\ref{key23}) holds.
	It remains to show that $c$ is an admissible function. Indeed, it follows from \eqref{keynew8} that $c(t)\geq 0$ for every $t\in D.$ Furthermore, we have
	\[
	0\leq c(0) = \inf\left\{f(y)-f(\bar x)\mid y\in D_{\bar x}\textrm{ and } T_M(y,\bar x)=0\right\}\leq f(\bar x)-f(\bar x)=0,
	\]
	which yields that condition (i) of Definition \ref{admisd} is fulfilled.
	To verify condition (ii), consider any sequence $(t_k)\subset D$ such that $c(t_k)\to 0$. Then it follows from the definition of $D$ and $c$ that for each $t_k$, there exists $x_k\in D_{\bar x}$ such that $T_M(x_k,\bar x)=t_k$ and $$0\leq f(x_k)-f(\bar x)<c(t_k)+\frac{1}{k}.$$ Since $c(t_k)+\frac{1}{k}\to 0$ as $k\to \infty$, we get $f(x_k)\to f(\bar x)$. Using the Tykhonov well-posedness with respect to $M$ of (MP) gives us $t_k=T_M(x_k,\bar x)\to 0$. This justifies (ii) of Definition \ref{admisd}. So, $c$ is an admissible function. 
	
	Conversely, if there are an $\bar x\in X$ and an admissible function $c$ satisfying $$f(\bar x)+c(T_M(x,\bar x))\leq f(x),$$ then we have
	\begin{equation} \label{key2}
		f(\bar x)\leq f(\bar x)+c(T_M(x,\bar x))\leq f(x) \textrm{ for all } x\in D_{\bar x}.
	\end{equation}
	This implies that $\bar x$ is global minimum point of $f$ with respect to $M$ on $D_{\bar x}$. To see the uniqueness of $\bar x$, we shall prove by contradiction. Assume that there exist $\bar y\in D_{\bar x}\setminus\{\bar x\}$ such that $f(\bar x)=f(\bar y)$. Then it follows from (\ref{key2}) that $$f(\bar x)+c(T_M(\bar y,\bar x))\leq f(\bar y),$$ and hence $c(T_M(\bar y, \bar x))=0$. Since $c$ is a admissible function, one has $T_M(\bar y,\bar x)=~0$ which yields $\bar y = \bar x$. Therefore, $\bar x$ is the unique global minimum point of $f$ with respect to $M$
	with respect to $M$ on $D_{\bar x}$. Moreover, if ($x_k)\subset D_{\bar x}$ is a sequence satisfying $f(x_k)\to f(\bar x)$, then $$c(T_M(x_k,\bar x))\leq(f(x_k)-f(\bar x))\to 0\text{ as }k\to\infty.$$ Using again the admission of $c$ gives $T_M(x_k,\bar x)\to 0$. Thus, problem (MP) is Tykhonov well-posed with respect to $M$.  The proof is complete.  
\end{proof}

Note that if $M=\uS$ then Theorem \ref{theo2} implies Theorem 12 in \cite{dontchev2006well} (for the case of Banach
spaces). 

The next example illustrates that usual global solutions do not exist for the problem, which means it is not Tykhonov well-posed in the classical sense. However, it may still have global solutions with respect to $M$, suggesting that it can be considered directional Tykhonov well-posed in those specific directions.

\begin{example}
	Let $X=C_{[0,1]}$ (the usual space of the real continuous functions on $[0, 1]$ with the maximum norm), and let $P$ be the set of all polynomials on $[0,1]$. Let $f\colon X\to \mathbb{R}$ defined by
	\[
	f(x) = \begin{cases}
		\max\limits_{s\in [0,1]} x(s) &\text{ if } x\in P, \\
		1 & \text{ if } x\not\in P.
	\end{cases}
	\]
	One can verify that $f$ is not bounded below and is not lower semicontinuous.
	
	Now, take $\bar x(s) \equiv 0$ for all $s\in[0,1]$ and $M = \{x(s)\equiv -2\}$. Then one has $$D_{\bar x} = \dom\, T_M(\cdot, \bar x) = \bar x - \cone (M) = \{x\mid x(s)\equiv \alpha \textrm{ for }\alpha\geq 0\}.$$ For convenience, we denote the constant polynomial $x(s)\equiv \alpha$ by $x_\alpha$ for all $\alpha\geq 0$. It is not hard to see that $f(x_\alpha) = \alpha$. For every $x_\alpha\in D_{\bar x}$, Proposition \ref{remark1}(ii) implies that there exists $u_\alpha\in M$ such that
	\[
	T_M(x_\alpha,\bar x) = \frac{\|x_\alpha - \bar x\|}{\|u_\alpha\|} = \frac{|\alpha-0|}{2} = \frac{\alpha}{2}.
	\]
	
	Next, let $c : D \to [0,\infty)$ be the identity map, i.e.,
	\[
	c(t) = t \quad \text{for all } t \in D,
	\]
	where $D = \{ T_M(x,\bar x) \mid x \in D_{\bar x} \}=[0,\infty).$ Clearly, $c$ is an admissible function. Moreover, for $x_\alpha \in D_{\bar x}$ we have
	\[
	f(\bar x) + c\!\left(T_M(x_\alpha,\bar x)\right) 
	= 0 + \tfrac{\alpha}{2} \leq \alpha = f(x_\alpha),
	\]
	which shows that inequality~\eqref{key23} holds. Therefore, Theorem~\ref{theo2} ensures that problem (MP) is Tykhonov well-posed with respect to $M$.
\end{example}

\subsection{Well-posedness under convexity assumptions}       	

Similar to results in \cite{durea2017new,long2022invariant,long2023directional}, we first present a general Ekeland variational principle in which
usual distance functions are replaced by minimal time functions mentioned in the previous sections. This principle will be an important tool in what
follows.

\begin{theorem}\label{ekel0}
	Let $f$ be as in the problem (MP). Suppose that $f$ is a lower semicontinuous function on $X$ and suppose further that there exists $\bar x \in X$ such that $f$ is bounded from below on $D_{\bar x}$. Then for every $\epsilon>0$ and every $x_0\in \mL(\bar x,\epsilon)=\{y\in D_{\bar x}\mid f(y)\leq \inf_{x\in D_{\bar x}} f(x)+\epsilon \}$, there exists $x_\epsilon\in D_{\bar x}$ such that the following inequalities hold:
	\begin{enumerate}
		\item [{\rm (i)}] $T_M(x_0,x_\epsilon) \leq \sqrt{\epsilon}$,
		\item [{\rm (ii)}] $f(x_\epsilon)+\sqrt{\epsilon} T_M(x_0,x_\epsilon)\leq f(x_0)$, and
		\item [{\rm (iii)}] $f(x_\epsilon) \leq f(y)+\sqrt{\epsilon} T_M(x_\epsilon,y)$ for all $y\in D_{\bar x}$.
	\end{enumerate}
\end{theorem}
\begin{proof} 
	Using the same arguments as in the proof of \cite[Proposition 3.3]{long2023directional}, we can deduce the conclusions from \cite[Theorem 3.1]{long2023directional}. 
\end{proof}

The following example illustrates that Theorem \ref{ekel0} can be viewed as a directional Ekeland variational principle.

\begin{example}\label{ex43}
	Let $X=C_{[0,1]}$ (the space of the real continuous functions on $[0, 1]$ with the maximum norm), and let $f\colon X\to \mathbb{R}$ defined by
	\[
	f(x) = 
	\min_{s\in [0,1]} x(s)    \]
	We can see that $f$ is continuous but is not bounded from below on $X$. Thus, a number of recent existing results are out of use.
	
	Now, take $\bar x(s) =s^2 $ for all $s\in[0,1]$ and $M = \{x\equiv -1\}$. Then one has $$D_{\bar x} = \dom\, T_M(\cdot, \bar x) = \bar x - \cone (M) = \{\bar x+\alpha\mid \alpha\geq 0\}.$$ 
	It is easy to see that $f$ is bounded from below on $D_{\bar x}$.
	Fixing any $\epsilon>0$ and  $x_0\in \mL(\bar x,\epsilon)=\{y\in D_{\bar x}\mid f(y)\leq \inf_{x\in D_{\bar x}} f(x)+\epsilon \}$, by Theorem \ref{ekel0}, there exists $x_\epsilon\in D_{\bar x}$ such that the inequalities in this theorem hold. 
\end{example}

\begin{remark} \label{vdkhac0} In Example \ref{ex43}, it should be emphasized that $(D_{\bar x}, T_M)$ does not define a metric space. In fact, take $x=\bar x+1$ and $y=\bar x+2$ in $D_{\bar x}$. Then by the definition of the minimal time function \eqref{minimaltime}, we have
	
	$$T_M(y,x)=\inf\{t\geq 0\mid \{\bar x+2- t\}\cap \{\bar x+1\}\neq \emptyset\}=1,$$
	and
	$$T_M(x,y)=\inf\{t\geq 0\mid \{\bar x+1- t\}\cap \{\bar x+2\}\neq \emptyset\}=\infty,$$
	where we use the convention that $\inf\emptyset=\infty.$ Thus $T_M(y,x)\neq T_M(x,y)$, which shows that the minimal time function $T_M$ is not symmetric, and hence cannot serve as a true distance function.
\end{remark}

Under additional assumptions on the data of problem (MP), we obtain the following characterization of Tykhonov well-posedness with respect to $M$.

\begin{theorem}\label{wpchar1} Consider the minimal time function \eqref{minimaltime} and the problem (MP). For every $\epsilon >0$, define
	$$ \mG(\epsilon):= \{x\in X\mid f(x)\leq f(y) + \epsilon T_M(y,x) \text{ for all }y\in D_x\}.$$
	Assume that $f$ is lower semicontinuous and convex. 
	Then (MP) is Tykhonov well-posed with respect to $M$ if and only if there exists $\bar x\in \bigcap_{\epsilon > 0}\mG(\epsilon)$ such that $f$ is bounded from below on $D_{\bar x}$ and $$\diam(\mG'(\bar x,\epsilon))\to 0\quad\text{as } \epsilon\to 0,$$ where
	$$\mG'(\bar x,\epsilon):= \{x\in D_{\bar x}\mid  f(x)\leq f(y)+\epsilon T_M(x,y) \text{ for all } y\in D_{\bar x}\}.$$
\end{theorem}
\begin{proof}
	Suppose that problem (MP) is directional Tykhonov well-posed with respect to~$M$. Then there exists $\bar x$ such that for all $y\in D_{\bar x}$ and all $\epsilon>0$, $$f(\bar x)\leq f(y) \leq f(y)+\epsilon T_M(y,\bar x),$$  which implies $\bar x\in \bigcap_{\epsilon > 0}\mG(\epsilon)$ and $f$ is bounded from below on $D_{\bar x}$. Next, we prove that $$\diam (\mG'(\bar x,\epsilon))\to 0\text{ as }\epsilon \to 0.$$ 
	Suppose on the contrary that there exist a sequence $(\epsilon_k)$ converging to $0$ and a constant $a>0$ such that $\diam(\mG'(\bar x, \epsilon_k))>2a$.
	Then one can find a sequence $(x_k)\subset \mG'(\bar x,\epsilon_k)$ satisfying 
	\begin{equation}\label{keynew2}
		\|x_k-\bar x\|\geq a
	\end{equation}
	and
	\begin{equation}\label{9}
		f(x_k)-f(\bar x)\leq \epsilon_k T_M(x_k,\bar x) \text{ for all }k\in\N.
	\end{equation}
	By Proposition \ref{remark1}(ii), for every $k\in \mathbb N$, there exists $u_k\in M$ such that
	\begin{equation}\label{keynew1}
		T_M(x_k,\bar x)= \frac{\|x_k-\bar x\|}{\|u_k\|}\geq \frac{a}{\|M\|}>0,
	\end{equation}
	where the inequality is fulfilled by \eqref{keynew2} and the definition of $\|M\|.$
	This ensures that the following set is nonempty:
	$$C:=\left\{x\in D_{\bar x}\mid T_M(x,\bar x)\geq \frac{a}{\|M\|}\right\}.$$
	Since (MP) is directional Tykhonov well-posed with respect to $M$, it follows that
	\begin{equation}\label{key6}
		m:= \inf\left\{f(x)-f(\bar x)\mid x\in C\right\}>0.
	\end{equation}
	Next, we show that 
	\begin{equation}\label{6}
		\frac{m}{a}T_M(x_k,\bar x) \leq f(x_k)-f(\bar x)\text{ for all } k\in \mathbb{N}.
	\end{equation} 
	Indeed, we first obtain from \eqref{keynew1} that
	\begin{equation}\label{key7}
		\lambda:= \frac{a}{\|u_k\|T_M(x_k,\bar x)}\leq1.
	\end{equation}
	Since $D_{\bar x}$ is convex, we have $\lambda  x_k + \left(1-\lambda \right)\bar x \in D_{\bar x}$ . Then using Proposition \ref{remark1}(ii), (\ref{key7}) and \eqref{keynew1} give us
	$$		T_M(\lambda  x_k + (1-\lambda )\bar x,\bar x)\geq
	\frac{\|\lambda  x_k + (1-\lambda )\bar x - \bar x\|}{\|M\|}
	= \frac{\lambda\|  x_k -\bar x\|}{\|M\|}=\frac{a}{\|M\|}.$$
	Consequently,
	\begin{equation}\label{8}
		\lambda  x_k + (1-\lambda )\bar x\in C.
	\end{equation}
	On the other hand, it follows from the convexity of $f$ that  
	$$f\left(\lambda  x_k + \left(1-\lambda \right)\bar x\right) - f(\bar x) \leq \lambda (f(x_k)-f(\bar x)).$$
	Combining this with (\ref{8}) and (\ref{key6}), we obtain
	\[
	m\leq \lambda (f(x_k)-f(\bar x)),
	\]
	which implies, by (\ref{key7}) and $\|u_k\|\geq 1,$ that the inequality (\ref{6}) holds. From (\ref{9}) and \eqref{6}, we get 
	$$	\frac{m}{a}T_M(x_k,\bar x) \leq f(x_k)-f(\bar x)\leq \epsilon_k T_M(x_k,\bar x)\text{ for all } k\in \mathbb{N}.$$
	Since $T_M(x_k,\bar x)>0$ (by \eqref{keynew1}), it follows that
	$$0<\frac{m}{a}\leq \epsilon_k\quad\text{for all }k\in\mathbb{N}.$$ 
	This is impossible since $\epsilon_k\to 0$ as $k\to 0$. Therefore, we conclude that $$\diam(\mG'(\bar x,\epsilon))\to 0\text{ as }\epsilon\to 0.$$
	
	Conversely, assume that there exists $\bar x\in \bigcap_{\epsilon > 0}\mG(\epsilon)$ such that $f$ is bounded from below on $D_{\bar x}$ and $$\diam(\mG(\bar x,\epsilon))\to~0\quad\text{as }\epsilon\to 0.$$ 
	First, since $\bar x\in \mG(\epsilon)$ for all $\epsilon >0$, we have
	\[
	f(\bar x)\leq f(y) + \epsilon T_M(y,\bar x)\quad \textrm{ for all } y\in D_{\bar x} \textrm{ and } \epsilon > 0.
	\]
	Letting $\epsilon\to 0$ gives
	\[
	f(\bar x)\leq f(y) \textrm{ for all }y\in D_{\bar x}.
	\]
	Therefore, $\bar x$ is is a minimizer of $f$ with respect to $M$ on $D_{\bar x}$.
	Next, we show that $$\mL(\bar x, \epsilon)\subset \left(\mG' (\bar x,\sqrt{\epsilon})\right)^{\sqrt{\epsilon}},$$ where $\mL(\bar x, \epsilon)$ is defined as in \eqref{eqthm1} and $$\left(\mG' (\bar x,\sqrt{\epsilon})\right)^{\sqrt{\epsilon}}:=\{x\in D_{\bar x}\mid T_M(x,\mG' (\bar x,\sqrt{\epsilon}))\leq \sqrt{\epsilon}\}.$$ 
	Indeed, picking any $x_0\in \mL(\bar x,\epsilon)$ we have $$f(x_0)\leq \inf_{x\in D_{\bar x}}f(x) + \epsilon = f(\bar x)+\epsilon.$$ 
	Applying Theorem \ref{ekel0}, there exists~$x_\epsilon\in D_{\bar x}$ such that 
	\begin{equation}\label{keynew3}
		T_M(x_0,x_\epsilon)\leq \sqrt{\epsilon}	
	\end{equation}
	and
	\begin{equation}\label{keynew4}
		f(x_\epsilon)\leq f(y) + \sqrt{\epsilon}T_M(x_\epsilon,x) \textrm{ for all }y\in D_{\bar x}.		
	\end{equation}
	Clearly, the relation \eqref{keynew4} implies that $x_\epsilon\in \mG'(\bar x,\sqrt{\epsilon})$. Combining this with \eqref{keynew3} gives  
	$$T_M\big(x_0,\mG'(\bar x,\sqrt{\epsilon})\big)\leq T_M(x_0,x_\epsilon)\leq \sqrt{\epsilon}.$$
	Therefore, we obtain $\mL(\bar x, \epsilon)\subset \left(\mG'(\bar x,\sqrt{\epsilon})\right)^{\sqrt{\epsilon}}$. By the imposed assumptions, it follows from Proposition \ref{lemmato0} that $\diam\left(\big(\mG'(\bar x,\sqrt{\epsilon})\big)^{\sqrt{\epsilon}}\right)\to 0$, which in its turn implies $$\diam\big(\mL(\bar x,\epsilon)\big)\to 0\text{ as }\epsilon\to 0.$$ 
	According to Theorem \ref{theo1}, problem (MP) is Tykhonov well-posed with respect to $M$. This completes the proof.
	
\end{proof}

\begin{remark}
	(i)	It is interesting to observe in general that 
	$$\lim_{\epsilon \to 0}\diam(\mG(\epsilon))\neq \lim_{\epsilon \to 0}\diam(\mG(\bar x, \epsilon))\text{ for every }\bar x\in \bigcap_{\epsilon >0}\mG(\epsilon)$$
	Indeed, consider $f: \mathbb R\to \R$ is given by $f(x)=e^x$. Take $M=\{-1\}$. Since $f$ is strictly increase, for every $x\in\R$ we get $$f(x)\leq f(y)\leq f(y)+\epsilon T_M(y,x)\text{ for all } y\in D_x=[x,\infty),$$
	which yields
	$\mG(\epsilon)=\R$ for any $\epsilon>0$. While if we take any $\bar x\in \bigcap_{\epsilon >0}\mG(\epsilon)$, then direct calculations yield
	$$\lim_{\epsilon \to 0}\diam(\mG(\bar x, \epsilon))=0.$$
	
	(ii) If $M=\mathbb{S}$, we get $T_M(x,y) = T_M(y,x) = \|x-y\|$ and $D_{\bar x} = X$. Hence, 
	\[
	\mG(\epsilon) = \mG'(\bar x,\epsilon) = \{x\in X\mid f(x)\leq f(y) + \epsilon\|x-y\|\text{ for all } y\in X\}.
	\] So, our Theorem \ref{wpchar1} implies Theorem 4.1 in \cite{lucchetti1981characterization} (for the case of a Banach space).
\end{remark}

The following example gives a case where Theorem \ref{wpchar1} can be applied, while a number
of earlier results cannot.

\begin{example} \label{ex1}
	Let $X=\ell^1= \{(x^{(1)},x^{(2)},...,x^{(i)},...) \mid x^{(i)} \in \R,\;	\sum_{i=1}^{\infty}|x^{(i)}|<\infty\}$ and let $f : X \to\R$ be defined by
	\[
	f(x) = \sum_{n=1}^{\infty}\frac{|\langle x,e_n\rangle|}{n},
	\]
	where $\langle\cdot,\cdot\rangle$ is the usual scalar product and $e_n = (0, 0,\cdots, 1, 0,\cdots),$ 1 at the $n$-th position, is the standard basis. Observe that the function $f$ is convex and continuous, and the
	problem (MP) has a unique solution at $ 0.$ But $(e_n)$ is a minimizing sequence which does not converge to $0.$ Thus, (MP) is not classically Tykhonov well-posed.
	
	Now let $M=\{(2,0,\ldots)\}$ and $\bar x = 0.$ Since $\bar x$ is a global minimum of $f$ on $X$, for any $\epsilon >0$, one has
	$$f(\bar x)\leq f(y)\leq f(y)+\epsilon T_M(y,\bar x)\quad\text{ for all }y\in D_{\bar x}.$$
	This implies that $$\bar x\in \bigcap_{\epsilon >0}\mG(\epsilon).$$ Clearly, $f$ is bounded from below on $D_{\bar x} = \{(-\alpha,0,\cdots)\in \ell^1\mid \alpha\geq 0\}$. To apply Theorem \ref{wpchar1}, it remains to show that $\diam (\mG'(0,\epsilon))\to 0$ as $\epsilon\to 0$. For any $\epsilon>0$ and $x:=(-\alpha_x,0,\cdots)\in D_{\bar x}$, if $x\in \mG'(0,\epsilon)$, then 
	\begin{equation}\label{eq3}
		\alpha_x=\sum_{n=1}^{\infty}\frac{|\langle x,e_n\rangle|}{n}=f(x)\leq f(y)+\epsilon T_M(x,y) \quad \textrm{ for all } y\in D_{\bar x}.
	\end{equation}
	Taking $y=0$ givess $f(y)=f(0)=0$ and
	\begin{equation}\label{keynew5}
		T_M(x,0)<\infty\quad  (\text{because }x\in D_{\bar x}=D_0).
	\end{equation}
	Then, (\ref{eq3}) becomes
	\[
	\alpha_x\leq \epsilon T_M(x,0)\quad\text{for all } x\in \mG'(0,\epsilon).
	\]
	By Proposition \ref{remark1}(ii), we have $2 T_M(x,0)=\|x-0\| = \alpha_x$, and hence the above inequality implies $$\diam (\mG'(0,\epsilon))\leq 2\epsilon T_M(x,0).$$ Combining this with \eqref{keynew5}, we obtain 
	$$\diam (\mG'(0,\epsilon))\to 0\quad\text{ as }\epsilon\to 0.$$ 
	Therefore, the sufficient conditions of Theorem~\ref{wpchar1} are fulfilled. By this theorem, problem (MP) is Tykhonov well-posed with respect to $M$.
\end{example}


\section{Relationships between level sets and admissible functions} \label{secrelasionship}

As seen in the previous section, level sets and admissible functions play a vital role in studying characterizations for the Tykhonov well-posedness with respect to $M$ of optimization problems. Thus, it is interesting to investigate properties and establish relationships between these level sets and admissible functions.

\begin{definition}\label{subhomo} 
	Let $g\colon X\to \mathbb{R}$ and $\bar x\in X$. The function $g$ is called \textit{subhomogeneous with respect to} $\bar x$ on $D_{\bar x}$ if for every $x\in D_{\bar x}$,
	\begin{equation}\label{subhomo1}
		g(sx +(1-s)\bar x)\leq sg(x) +(1-s)g(\bar x) \textrm{ for all } s\in [0,1].
	\end{equation}
	
\end{definition}

Consider the problem (MP) and the minimal time function \eqref{minimaltime}. Throughout this section, we impose the following assumptions:
\begin{enumerate}
	\item [{\rm (i)}] $M$ is a subset of the unit sphere, that is, $M \subset \mathbb S$.
	\item [{\rm (ii)}] $\bar x\in X$ is a global minimum point of $f$ with respect to $M$ on $D_{\bar x}$.
	\item [{\rm (iii)}] $f$ is subhomogeneous with respect to $\bar x$ on $D_{\bar x}$, which is a weaker condition compared to convexity.
\end{enumerate}

\begin{remark}
	If $\bar x = 0$, $f(0) = 0$ and $M=\uS$, then (\ref{subhomo1}) becomes $f(sx)\leq sf(x)$ for all $s\in[0,1]$ and $x\in X$, then $f$ is called subhomogeneous with respect to $0$; see \cite[pp. 463]{lucchetti1981characterization}.
\end{remark}

For $\epsilon >0$, denote by	$$\mL(\bar x,\epsilon):=\left\{y\in D_{\bar x}\mid f(y)\leq \inf_{x\in D_{\bar x}}f(x)+\epsilon\right\},$$ 
the level set defined in Theorem \ref{theo1}, and put
\[
r(\epsilon) = \frac{1}{2}\diam (\mL(\bar x,\epsilon)).
\]

For every $t\in [0,\infty),$ define
\[
c_0(t) = \inf\{f(x)-f(\bar x)\mid x\in D_{\bar x}\text{ and } T_M(x,\bar x)=t\}.
\]

\begin{remark} 	Under the additional assumption on $M$,	it follows from Proposition \ref{pronew} that 
	$$\{T_M(x,\bar x)\mid x\in D_{\bar x}\}=[0,\infty).$$
	Therefore, $c_0(t)<\infty$ for all $t\in [0,\infty).$
	Moreover, since $\bar x$ is a global minimum point of $f$ with respect to $M$ on $D_{\bar x}$, we clearly have  $c_0(t)\geq 0$ for all $t\geq0$, and in particular $$c_0(0)=f(\bar x)-f(\bar x)=0.$$ 
\end{remark}

We now proceed to establish an elementary estimate for the function $c_0$.

\begin{proposition} \label{theo3} Let $t_0> 0$. Then for every $s\in [0,t_0]$, we have 
	\begin{equation}\label{estimate1}
		c_0(s)\leq \frac{s}{t_0}c_0(t_0).
	\end{equation}
\end{proposition}
\begin{proof}
	We have by Proposition \ref{pronew} and the definition of $c_0$ that for every $\epsilon >0$, there exists an $x_\epsilon\in D_{\bar x}$ such that 
	\begin{equation}\label{keynew29}
		T_M(x_\epsilon,\bar x)=t_0
	\end{equation}
	and
	\begin{equation}\label{key8}
		f(x_\epsilon) - f(\bar x)\leq c_0(t_0)+\epsilon.    
	\end{equation}
	Now setting $x'_\epsilon:= \frac{s}{t_0}x_\epsilon + \left(1-\frac{s}{t_0}\right)\bar x$ gives $x'_\epsilon\in D_{\bar x}$ (since $D_{\bar x}$ is convex). Thus, it follows from Proposition \ref{remark1}(ii) and (\ref{keynew29}) that 
	$$ \|    T_M(x'_\epsilon,\bar x)\|=\|x'_\epsilon-\bar x\|=\left\|\frac{s}{t_0}x_\epsilon + \left(1-\frac{s}{t_0}\right)\bar x-\bar x\right\|=s.
	$$
	Using again the definition of $c_0$, we get
	\begin{equation}\label{eq44}
		c_0(s)\leq f(x'_\epsilon) - f(\bar x) = f\left(\frac{s}{t_0}x_\epsilon + \left(1-\frac{s}{t_0}\right)\bar x\right) - f(\bar x).    
	\end{equation}
	On the other hand, since $f$ is subhomogeneous with respect to $\bar x$ on $D_{\bar x}$, it follows from (\ref{key8}) that  for all $\epsilon >0$,
	\begin{equation*}
		f\left(\frac{s}{t_0}x_\epsilon + \left(1-\frac{s}{t_0}\right)\bar x\right) - f(\bar x)\leq \frac{s}{t_0}\Big(f(x_\epsilon)-f(\bar x)\Big)\leq \frac{s}{t_0}\Big(c_0(t_0)+\epsilon \Big).   
	\end{equation*}
	Combining this with (\ref{eq44}) gives us $$c_0(s)\leq \frac{s}{t_0}(c_0(t_0)+\epsilon).$$ Since $\epsilon$ could be taken arbitrarily small, we obtain the desired inequality \eqref{estimate1}, which therefore completes the proof of the proposition.
\end{proof}

\begin{remark}
	As noted in Remark~\ref{vdkhac} and Example \ref{vdratkhac}, one may have $$T_M(x,\bar x) \neq T_M(\bar x,x).$$ 
	Therefore, Proposition~\ref{theo3} generalizes Proposition~5.1 of \cite{lucchetti1981characterization} 
	from the usual distance to the directional minimal time function with respect to a subset of the unit sphere (for the case of Banach spaces).
\end{remark}

The following facts are useful in what follows.

\begin{remark}\label{adm1}
	If $0<c_0(t_0)$ for some $t_0>0$, then $c_0$ is a strictly increasing function on $[t_0,\infty)$. Indeed,
	for any $a,b\in [t_0,\infty)$ satisfying $a<b$, Proposition \ref{theo3} tells us that
	\[
	c_0(a)\leq \frac{a}{b}c_0(b) \text{ and } 0<c_0(t_0)\leq \frac{t_0}{b}c_0(b),
	\]
	and hence $c_0(b)>0$. Moreover, since $\frac{a}{b}<1$ we get
	\[
	c_0(a)\leq \frac{a}{b}c_0(b)<c_0(b).
	\]
\end{remark}

\begin{remark}\label{adm2}
	If $c_0(t)>0$ for all $t>0$, then $c_0$ is an admissible function. Indeed, we only need to show that if $c_0(t_k)\to 0$, then $t_k\to 0$. Suppose for contradiction that $t_k\not\to 0$. Without loss of generality we can assume that $t_{k}>\epsilon$ for some $\epsilon>0$. By Proposition \ref{theo3}, we have $c_0(\epsilon)\leq \frac{t_{k}}{\epsilon}c(t_{k})$. Moreover, since $c_0(t_k)\to 0$ as $k\to \infty$, and by the assumption one has $0<c_0(\epsilon)$. Then it follows from Remark~\ref{adm1} that $0<c_0(\epsilon)< c_0(t_{k})$, which is impossible as $c_0(t_k)\to 0$. So $c_0$ is admissible whenever $c_0(t)>0$ for all $t>0$. 
\end{remark}

The next theorem provides an estimate for $r(c_0)$, extending Propositions~5.5 and~5.8 of \cite{lucchetti1981characterization} to the framework of directional minimal time functions in Banach spaces.

\begin{theorem} \label{level1}
	For every $t>0$, if $0<c_0(t)$ then following inequalities hold:
	\begin{equation}\label{eq11}
		\frac{t}{2}\leq r(c_0(t))\leq t.
	\end{equation}        
\end{theorem}
\begin{proof}
	Let $t>$ be such that $0<c_0(t)$. Observe that the first inequality of (\ref{eq11}) is equivalent to
	\begin{equation}\label{key9}
		t\leq \diam\big(\mL(\bar x, c_0(t)\big).
	\end{equation}
	By Proposition \ref{pronew} and \ref{remark1}(ii), together with the definition of $c_0$, for every $\epsilon>0$ there exists $x_\epsilon\in D_{\bar x}$ such that
	\begin{equation}\label{keynew33}
		\|x_\epsilon-\bar x\|=\|T_M(x_\epsilon,\bar x) = t
	\end{equation}
	and 
	\begin{equation}\label{eq12}
		f(x_\epsilon)-f(\bar x)\leq c_0(t)+x_\epsilon.            
	\end{equation}
	Since $f$ is subhomogeneous with respect to $\bar x$ on $D_{\bar x}$, for any $s\in (0,t)$ we have
	$$
	\begin{array}{ll}
		f\left(\frac{t-s}{t} x_\epsilon +\frac{s}{t}\bar x \right) -f(\bar x)&\leq \frac{t-s}{t}f(x_\epsilon) +\frac{s}{t}f(\bar x)-f(\bar x)\\
		&= \frac{t-s}{t}(f(x_\epsilon)-f(\bar x)).    
	\end{array}$$
	By choosing $0<\epsilon\leq \frac{s}{t-s}c_0(t)$, the above inequality and (\ref{eq12}) give us that
	$$\begin{array}{ll}
		f\left(\frac{t-s}{t}x_\epsilon +\frac{s}{t}\bar x \right) -f(\bar x)&\leq \frac{t-s}{t}(f(x_\epsilon)-f(\bar x))\\
		&\leq \frac{t-s}{t}\left(c_0(t)+\epsilon\right)\\
		&\leq \frac{t-s}{t}\left(c_0(t)+\frac{s}{t-s}c_0(t)\right)=c_0(t).
	\end{array}	$$
	Thus, $x_s\in \mL(\bar x,c_0(t))$, where $x_s:=(\frac{t-s}{t}) x_\epsilon + (\frac{s}{t}) \bar x\in D_{\bar x}$. Moreover, one has
	$$
	\begin{array}{ll}
		\|x_s-\bar x\| &= \left\|\frac{t-s}{t} x_\epsilon - \frac{t-s}{t} \bar x\right\|\\
		&= \frac{t-s}{t}\| x_\epsilon - \bar x\|\\
		& = t-s,
	\end{array}$$
	where the last equality is fulfilled by \eqref{keynew33}.
	Thus, for any $s\in (0,t)$, there exists $x_s\in \mL(\bar x, c_0(t))$ such that $\|x_s-\bar x\| = t-s.$ This implies $$\diam (\mL(\bar x,c_0(t)))\geq t-s\text{ for all }s\in (t,0).$$ Letting $s\to 0$, we obtain the inequality (\ref{key9}), which therefore completes the proof of the first inequality of (\ref{eq11}).
	
	For the second inequality of (\ref{eq11}), we argue by way of contradiction. Assume that there exists $t_0>0$ such that $0<c(t_0)$ and $r(c_0(t_0))>t_0$. Equivalently, $\diam\big(\mL(\bar x,c_0(t_0))\big)>2t_0$. Then there exists an $x'\in \mL(\bar x,c_0(t_0))$ such that $$T_M(x',\bar x) =\|x'-\bar x\|=t' > t_0,$$ where the first equality is fulfilled by Proposition \ref{remark1}(ii). This together with the definition of $c_0$ ensures  
	\[
	c_0(t')\leq f(x') - f(\bar x) \leq c_0(t_0).
	\]
	However, since $0<c_0(t_0)$ and $t_0<t'$, Remark \ref{adm1} gives $c_0(t_0)<c_0(t')$, a contradiction. This completes the proof of the second inequality of (\ref{eq11}) and, therefore, of the whole theorem. 
\end{proof}

We end this section with a result that highlights  the relationship between admissible functions and level sets in terms of directional minimal times functions.

\begin{theorem} \label{theo5}
	The function $c_0$ is admissible if and only if $\diam\big(\mL(\bar x, \epsilon)\big)\to 0$ as $\epsilon\to 0$.
\end{theorem}
\begin{proof}
	Assume that $c_0$ is admissible. By Definition \ref{admisd} and Proposition \ref{theo3}, one has $c_0(t)>0$ for all $t>0$. Two cases are possible:  
	
	(A) there exists an $\epsilon_0 > 0$ such that $c_0(t)>\epsilon_0$ for all $t>0$ and 
	
	(B) for any sequence $(\epsilon_k)$ satisfying $\epsilon_k\rightarrow 0$, there exists a sequence $(t_k)$ such that $t_k>0$ and $c_0(t_k)\leq \epsilon_k$ for all $k\in \N$. 
	
	Let us show below that in each of these cases, we arrive at the conclusion of the first implication. In case (A), one has $$f(x)-f(\bar x)>\epsilon_0\quad\text{for all }x\in D_{\bar x}\text{ satisfying }T_M(x,\bar x)>0.$$ 
	Thus, $\mL(\bar x, \delta)=\{\bar x\}$ for all $\delta \in (0,\epsilon_0)$, which implies that $\diam\big(\mL(\bar x, \epsilon)\big)\to 0$ as $\epsilon\to 0$. 
	In case (B), one has
	\begin{equation}\label{eq9}
		\diam\big(\mL(\bar x,\epsilon_k)\big) = 2r(\epsilon_k)\leq 2r\big(c_0(t_k)\big)\leq 2t_k,
	\end{equation}
	where  the first inequality holds because $r$ is a decreasing function and the second inequality is satisfied by Theorem \ref{level1}. 
	Notice that since the function $c_0$ is admissible, we deduce from the definition that $t_k\to 0$.
	Combining this with (\ref{eq9}), we obtain $\diam\big(\mL(\bar x,\epsilon_k)\big)\rightarrow 0$ for any sequence $\epsilon_k\rightarrow 0$. 
	
	Conversely, assume that $\diam\big(\mL(\bar x,\epsilon)\big)\rightarrow 0$ as $\epsilon\rightarrow 0$. By Remark \ref{adm2}, we only need to show that $$c_0(t)>0\quad\text{for all }t>0.$$ Indeed, if there exists $t_0>0$ such that $c_0(t_0)=0$, then for every $k\in \N$ we can find $x_k\in D_{\bar x}$ such that $$T_M(x_k,\bar x)=t_0\quad\text{and}\quad f(x_k)-f(\bar x)<1/k.$$ This means $x_k\neq \bar x$ and $x_k\in \mL(\bar x, 1/k)$ for all $k\in \N$. However, since $\bar x\in \mL(\bar x, 1/k)$ for all $k\in \N$, this contradicts the assumption that $\diam\big(\mL(\bar x,\epsilon)\big)\rightarrow 0$ as $\epsilon\to 0$. Therefore, we conclude that the function $c_0$ is admissible. 
\end{proof}




\section*{Acknowledgements}
The first author is supported by National Science and Technology Council (NSTC 112-2115-M-003-012-MY2, NSTC 113-2124-M-003-002); Higher Education Sprout Project-Center for Optimal Intelligent Data Analytics and Prediction. The third author is supported by Vietnam National Foundation
for Science and Technology Development (NAFOSTED) grant 101.01-2025.14

\section*{Disclosure statement}
No potential conflict of interest was reported by the author(s).


\end{document}